\documentclass[a4paper,11pt]{amsart}

\usepackage[T1]{fontenc}

\usepackage{enumerate}

\newtheorem{theorem}{Theorem}
\newtheorem{proposition}{Proposition}[section]

\newtheorem{lemma}[proposition]{Lemma}
\theoremstyle{definition}
\newtheorem{definition}[proposition]{Definition}

\newtheorem{openproblem}{Open problem}
\theoremstyle{remark}

\numberwithin{equation}{section}

\newcommand{\R}{\mathbf{R}}
\newcommand{\Hei}{\mathbf{H}}
\newcommand{\st}{\: :\:}
\newcommand{\dualprod}[2]{\langle {#1},{#2}\rangle}
\newcommand{\abs}[1]{\lvert {#1}\rvert}
\newcommand{\Norm}[1]{\lVert {#1}\rVert}
\newcommand{\dif}{\,\mathrm{d}}
\DeclareMathOperator{\Div}{div}
\DeclareMathOperator{\Sym}{Sym}
\DeclareMathOperator{\Ad}{Ad}

\title{Subelliptic Bourgain--Brezis estimates on groups}

\author{Sagun Chanillo}
\address{Department of Mathematics -- Hill Center\\
Rutgers, The State University of New Jersey\\
110 Frelinghuysen Rd.\\
Piscataway, NJ 08854-8019}
\email{chanillo@math.rutgers.edu}

\author{Jean Van Schaftingen}

\address{D\'epartment de Math\'ematique\\
Universit\'e catholique de Louvain\\
Chemin du Cyclotron 2\\
1348 Louvain-la-Neuve\\
Belgium}
\email{Jean.VanSchaftingen@uclouvain.be}

\date{\today}



\begin{document}

\begin{abstract}
We show that divergence-free $\mathrm{L}^1$ vector fields on a homogeneous group of homogeneous dimension $Q$ are in the dual space of functions whose gradient is in $\mathrm{L}^Q$. This was previously obtained on $\R^n$ by Bourgain and Brezis.
\end{abstract}

\maketitle

\tableofcontents

\section{Introduction}

On $\R^n$, the Sobolev embedding theorem states that $W^{1,q}(\R^n) \subset \mathrm{L}^{\frac{nq}{n-q}}(\R^n)$ when $q < n$. When $q=n$, the embedding of $W^{1,n}(\R^n)$ in $\mathrm{L}^\infty(\R^n)$ that is suggested by homogeneity arguments of the critical Sobolev space is known to fail. 
However, maps in the critical Sobolev space have many properties in common with bounded or continuous maps. An example of such a property was obtained by Bourgain and Brezis, who showed that divergence-free vector fields see $W^{1,n}$ functions as if they were bounded functions \cite{BourgainBrezis2004,BourgainBrezis2007}; that is, if $F \colon \R^n \to \R^n$ is a divergence-free vector field, one has
\begin{equation}
\label{ineqBourgainBrezis}
 \Bigl\lvert\int_{\R^n} \varphi \cdot F \Bigr\rvert \le C \Norm{F}_{\mathrm{L}^1(\R^n)} \Norm{\nabla \varphi}_{\mathrm{L}^n(\R^n)}\;.
\end{equation}
A striking consequence of this fact is that if $U$ is the solution
\[
  -\Delta U=F\;,
\]
given by convolution with the Newton kernel, then $\nabla U \in \mathrm{L}^{n/(n-1)}(\R^n)$, whereas without the condition on the divergence, the best result that can be obtained is that $\nabla U$ belongs to weak $\mathrm{L}^{n/(n-1)}$, the Marcinkiewicz space $\mathrm{L}^{n/(n-1),\infty}(\R^n)$.
The proof of Bourgain and Brezis relies on a Littlewood--Paley decomposition, and yields in fact a necessary and sufficient condition on the divergence of an $\mathrm{L}^1$ vector field for this vector field to induce a linear functional on the homogeneous Sobolev space $\mathrm{\dot{W}}^{1,n}(\R^n)$. 
The estimates on $\R^n$ can be transported on smooth domains \cite{BrezisVanSchaftingen2007} or on manifolds. 

The present paper starts from the question whether similar estimates hold without the local structure of the commutative group $\R^n$, and gives an answer on homogeneous groups. A homogeneous group $G$ is a connected and simply connected Lie group such that the Lie algebra $\mathfrak{g}$ of left-invariant vector fields is a graded, nilpotent and stratified Lie algebra, that is
\begin{enumerate}
\item $\mathfrak{g}=V_1 \oplus V_2 \oplus \dots \oplus V_p$,
\item $[V_i,V_j] \subset V_{i+j}$ for $i+j \le p$ and $[V_i,V_j]=\{0\}$ if $i+j > p$,
\item $V_1$ generates $\mathfrak{g}$ by Lie brackets.
\end{enumerate}
While the dimension of $G$ as a manifold is $n=\sum_{j=1}^p m_j$,
where $m_j =\dim V_j$, the homogeneous dimension $Q=\sum_{j=1}^p j m_j$ plays an essential role. In particular when $q < Q$, it was shown that \cite{Folland1975,FollandStein1974,RotschildStein1976}
\[
  S^{1,q}(G)=\{ u \in \mathrm{L}^q(G) \st Y_iu \in \mathrm{L}^q(G) \text{ for }1 \le i \le m \} \subset \mathrm{L}^{\frac{Qq}{Q-q}}(G)\;,
\]
where $\{Y_i\}_{i=1}^{m}$ is a basis of $V_1$ and the measure used to define $\mathrm{L}^q(G)$ is the left- and right-invariant Haar measure $\mu$ on $G$.

In this paper, we show that functions in $\mathrm{\dot{S}}^{1,Q}(G)$ are seen like bounded functions by divergence-free $\mathrm{L}^1$ vector fields. Before defining these, we define the bundle $T_bG$ by restricting the vectors to be in $V_1$. The vector-field $F$ is divergence-free if
\[
 \int_{G} F \psi \dif  \mu=0\;,
\]
for every compactly supported smooth function $\psi \in C^\infty_c(G)$. We finally use the notation $\nabla_b u=(Y_1 u,\dots,Y_{m}u)$. Our main result is:

\begin{theorem}
\label{mainTheorem}
If $\varphi \in C^\infty_c(G,T^*_bG)$ is a section of the cotangent bundle and the vector field $F \in \mathrm{L}^1(G;T_bG)$ is divergence-free, then
\[
  \Bigl\lvert\int_{G} \dualprod{\varphi}{F} \dif \mu \Bigr\rvert\le C \Norm{F}_{\mathrm{L}^1(G)}\Norm{\nabla_b \varphi}_{\mathrm{L}^Q(G)}\;.
\]
\end{theorem}

The proof uses the strategy developed by the second author to give an elementary proof of \eqref{ineqBourgainBrezis} \cite{VanSchaftingen2004}. 
That proof relied on splitting the integral on hyperplanes, and using H\" older continuity of the restriction of $\varphi$ on hyperplanes. One could then split $\varphi$ into one part which is bounded and another whose gradient is bounded.
The estimate on the latter relied on the divergence-free condition. One concluded then by H\" older's inequality. 

In the setting of homogeneous groups, hyperplanes are replaced by cosets of codimension $1$ normal subgroups. 
While on $\R^n$ the splitting of $\varphi$ on hyperplanes only used derivatives of $\varphi$ in directions parallel to the hyperplane,  on a homogeneous group using only the directions of $V_1$ parallel to the normal subgroups is not sufficient to have the right esimates for the splitting. In order to circumvent this problem, our spliting relies on information about all the derivatives of $\varphi$ in some neighbourhood of the normal subgroup. 
The splitting estimates are then obtained with Jerison's machinery for analysis on homogeneous groups \cite{Jerison1986}, and depend now on some maximal function associated to $\varphi$.

As a consequence of Theorem~\ref{mainTheorem}, we give a regularity result for the subelliptic Laplacian
$\Delta_b=\sum_{i=1}^{m} Y_i^2$.

\begin{theorem}
\label{theoremRegularity}
If $F \in \mathrm{L}^1(G,T_bG)$ is divergence-free, then the problem
\[
  -\Delta_b U=F
\]
has a solution $U \in \mathrm{\dot{S}}^{1,Q/(Q-1)}$ satisfying the estimate
\[
  \Norm{\nabla_b U}_{\mathrm{L}^{Q/(Q-1)}} \le C \Norm{F}_{\mathrm{L}^1(G)}\;.
\]
\end{theorem}

On $\R^n$, Bourgain and Brezis have also proved that if $F \in \mathrm{L}^1(\R^n;\R^n)$, one has $F \in \mathrm{\dot{W}}^{-1,n/(n-1)}(\R^n)$ if and only if $\Div f \in \mathrm{\dot{W}}^{-2,n/(n-1)}(\R^n)$. In view of Theorem~\ref{mainTheorem}, we ask the question whether their result extends to homogeneous groups.

\begin{openproblem}
Let $F \in \mathrm{L}^1(G;T_bG)$ be a vector field. Does one have $F \in \dot{\mathrm S}^{-1,Q/(Q-1)}(G;T_bG)$ if and only if $\Div_b F \in \dot{\mathrm S}^{-2,Q/(Q-1)}(G)$?
\end{openproblem}

The rest of this paper is organized as follows. In section~\ref{sectionApproximation}, we state and prove Lemma~\ref{lemmaApproximation} about the approximation of a function $u \in \mathrm{\dot{S}}^{1,Q}(G)$ on a normal subgroup $G_i$ of $G$. This lemma is the main new ingredient for the proof of Theorem~\ref{mainTheorem}, which is the object of section~\ref{sectionMainProof}. In a short section~\ref{sectionRegularity}, we show how the combination of Theorem~\ref{mainTheorem} with classical regularity estimates on homogeneous groups leads to Theorem~\ref{theoremRegularity}. In the last section~\ref{sectionFurther}, we give generalizations of Theorem~\ref{mainTheorem} in several directions: $\mathrm{L}^1$--divergence vector fields, critical fractional  Sobolev spaces, and higher order conditions.

The research of S.C.\ was supported in part by a grant from the NSF; the research of J.V.S.\ was supported in part by a grant of the Fonds de la Recherche Scientifique--FNRS.

\section{Approximation on normal subgroups}

\label{sectionApproximation}

In order to prove Theorem~\ref{mainTheorem}, we slice $G$ into cosets of codimension $1$ normal subgroups that are  constructed as follows. Fix $1\le i \le m$, let $\mathfrak{g}_i$ be the linear space spanned by $\{Y_j\}_{j \ne i}$ and by $V_\ell$, $2 \le \ell \le p$, and let $G_i$ be the image of $\mathfrak{g}_i$ by the exponential map. Since $\mathfrak{g}$ is graded, $\mathfrak{g}_i$ is an ideal of $\mathfrak{g}$, and $G_i$ is a normal subgroup of $G$. Since $G$ is simply-connected, one has $G/G_i \cong \R$. The Haar measure $\nu$ on $G_i$ is normalized so that
\[
 \mu(A)=\int_\R \nu(G_i \cap e^{-tY_i}A)\dif t\;.
\]

\begin{lemma}
\label{lemmaApproximation}
There exists $C >0$ such that, for every $u \in C^\infty(G)$, $\lambda > 0$ and $1 \le i \le m$, there exists $u_\lambda \in C^\infty(G)$ such that 
\begin{align}
  \label{ineqLinftyApprox} \Norm{u-u_\lambda}_{\mathrm{L}^\infty(G_i)} &\le C \lambda^{\frac{1}{Q}}M(I)(0)\;,\\
  \label{ineqLipschitzApprox} \Norm{\nabla_b u_\lambda}_{\mathrm{L}^\infty(G)} &\le C \lambda^{\frac{1}{Q}-1} M(I)(0)\;,
\end{align}
where
\[
  I(t)=\Bigl(\int_{G_i} \abs{\nabla_b u(e^{tY_i}h)}^Q\dif \nu(h)\Bigr)^\frac{1}{Q}
\]
and $M(I)$ is the Hardy--Littlewood maximal function of $I$.
\end{lemma}

The proof of this Lemma relies on several tools developed by Jerison for the analysis on Lie groups \cite{Jerison1986}. First, let $R$ denote the composition by the inverse: $Ru(g)=u(g^{-1})$. If $Y$ is a vector field, then the vector field $Y^R$ is defined by $Y^Ru= RY\!Ru$, where $Rg=g^{-1}$. If $Y$ is a left-invariant vector field on $G$, then $Y^R$ is a right-invariant vector field on $G$. 
The group convolution on $G$ is defined by
\[
  (u \ast v)(g)=\int_{G} u(gh^{-1})v(h)\dif  \mu(h)=\int_{G} u(h)v(h^{-1}g)\dif 
\mu(h)\;.
\]
From the associative law, if $Y$ is a left-invariant vector field, one has
\[
  Y(u \ast v)=u \ast Yv\;,
\]
and
\begin{equation}
\label{eqYRconv}
  (Yu) \ast v=-u \ast Y^Rv\;.
\end{equation}

One can define dilations on $G$. First define its derivative at the identity $d_\tau \colon \mathfrak{g}\to \mathfrak{g}$ by $d_\tau x=\tau^ix$ on $V_i$, for $1 \le i \le p$. One checks that $d_\tau$ is in an automorphism of $\mathfrak{g}$ as a Lie algebra. Therefore, the dilation $\delta_\tau \colon G \to G$ can be defined as the group automorphism such that the differential of $\delta_\tau$ at the identity is $d_\tau$. Note that $\mu(\delta_\tau A)=\tau^{Q} \mu(A)$.
For $\eta \colon G \to \R$, one further defines
\[
 I_\tau \eta (g)=\frac{1}{\tau^Q} \eta (\delta_{\tau^{-1}} g )\;,
\]
so that, if $\eta \in \mathrm{L}^1(G)$,
\[
  \int_G \eta\dif \mu=\int_G I_\tau \eta\dif \mu\;.
\]
The dilation also allows to define balls. Take the unit ball $B(e,1)$ around the identity $e$ to be the image of an euclidean ball on $\mathfrak{g}$ by the exponential, and define $B(g,\lambda)=g \delta_\lambda B(e,1)$.

The adjoint representation $\Ad : G \to GL(\mathfrak{g})$ is defined as follows: $\Ad(h)$ is the derivative of the automorphism $g \mapsto hgh^{-1}$. 
One has
\[
  [\Ad(h)Y] u (g)=\frac{\partial}{\partial t} u(g he^{tY}h^{-1}) \Bigr\vert_{t=0}
\]
and
\[
  \Ad( \delta_t h)Y=t^{-1}\delta_t  \Ad (h) Y\;.
\]
Since $\mathfrak{g}$ is nilpotent, one also has
\begin{multline*}
  \Ad ({e^{X}}) Y=Y+[X,Y]+\frac{1}{2}[X,[X,Y]]+
\dotsb\\+\frac{1}{(p-1)!}[X,[X,\dotsc [X,[X,Y]]\dotsc]]\;.
\end{multline*}

Finally, we need to transform some derivatives into derivatives with respect to right-invariant vector fields\;.

\begin{lemma}[Jerison \cite{Jerison1986}]
\label{lemmaJerison31p}
There exist differential operators $D^{(k)}$ such that for every $\eta \in C^\infty_c(G)$, 
\[
  \frac{\partial}{\partial \tau} I_\tau \eta=\sum_{k=1}^{m} Y_k^R I_\tau D^{(k)}\eta\;.
\]
For every $Y \in \mathfrak{g}$ there exist differential operators $D^{(k)}_{j}$ such that for every $\eta \in C^\infty_c(G)$, 
\[
  \delta_t Y I_\tau \eta=\sum_{j=1}^p \sum_{k=1}^{m} \bigl(\frac{t}{\tau}\bigr)^jY_k^R I_\tau D^{(k)}_j\eta\;.
\]
\end{lemma}
\begin{proof}
The first statement is exactly (b) of Lemma 3.1$'$ in Jerison's paper \cite{Jerison1986}. For the second statement, let $Y =\sum_{i=1}^p Y^j$ with $Y^j \in V_j$. Part (a) in the same Lemma 3.1$'$ \cite{Jerison1986} states that 
\[
  Y^j \eta=\sum_{k=1}^m Y_k^R D^{(k)}_j \eta\;,
\]
for some differential operators $D^{(k)}_j$. Since 
\begin{align*}
  I_\tau(Y^j \eta)&=\tau^j Y^j I_\tau \eta & \text{ and }& &I_\tau Y_k^R D^{(k)}_i \eta&=\tau  Y_k^RI_\tau D^{(k)}_i \eta\;,
\end{align*}
our statement follows immediately.
\end{proof}

\begin{proof}[Proof of Lemma~\ref{lemmaApproximation}]
Choose $\eta \in C^\infty_c(G)$ such that $\int_G \eta \dif \mu=1$. For every $g \in G_i$ and $t \in \R$, define
\[
  u_\lambda(ge^{tY_i})=(u \ast I_{\sqrt{\lambda^2+t^2}} \eta)(g)
\]
Let us first check \eqref{ineqLinftyApprox}. We need to estimate $\abs{u_\lambda(g)-u(g)}$ for $g \in G_i$. One has clearly
\[
  u_\lambda(g)-u(g)=\int_0^{\lambda} \frac{\partial }{\partial \tau}\Bigl[ u \ast I_\tau \eta\Bigr](g) \dif \tau
                =\int_0^{\lambda} \Bigl[u \ast \frac{\partial }{\partial \tau} I_\tau \eta\Bigr] (g)\dif \tau\;.
\]
Therefore,
\[
\begin{split}
  u_\lambda(g)-u(g)&=\sum_{k=1}^{m} \int_0^{\lambda} \Bigl[u \ast (Y_k^R I_\tau \eta^{(k)})\Bigr](g)\dif \tau \\
  &=-\sum_{k=1}^{m} \int_0^{\lambda} \Bigl[(Y_k u) \ast (I_\tau \eta^{(k)})\Bigr](g)\dif \tau\;,
\end{split}
\]
where $\eta^{(k)}=D^{(k)} \eta$ was provided by Lemma~\ref{lemmaJerison31p}, and \eqref{eqYRconv} justified the integration by parts.
Therefore, for some $C,K < \infty$,
\[
\begin{split}
  \abs{u_\lambda(g)-u(g)} &= \biggl\lvert\sum_{k=1}^{m}\int_0^\lambda \int_{G} Y_k u(h)I_\tau \eta^{(k)}(h^{-1}g)\dif \mu(h) \dif \tau \biggr\rvert\\
&\le C \int_0^{\lambda} \frac{1}{\tau^Q} \biggl(\int_{B(g,K\tau)} \abs{\nabla_b u(h)}\dif h\biggr)\dif \tau\;.
\end{split}
\]
Now note that $B(g,K\tau) \cap e^{tY_i}G_i = \emptyset$ 
when 
$\abs{t} 
\ge 
\kappa \tau$, 
for some $\kappa < \infty$; therefore
\[
  \abs{u_\lambda(g)-u(g)} 
\le C \int_0^{\lambda} \frac{1}{\tau^Q} \int_{]-\kappa \tau,\kappa \tau[} \int_{G_i \cap e^{-tY} B(g,K\tau)} \abs{\nabla_b u(e^{tY_i}h)}\dif \nu (h) \dif t\dif \tau\;.
\]
Since $\nu(e^{-tY} B(g,K\tau) \cap G_i)\le C s^{Q-1}$, we obtain, by H\"older's inequality,
\[
\begin{split}
  \abs{u_\lambda(g)-u(g)} 
&\le C'\int_0^{\lambda} \tau^{\frac{1}{Q}-1}\frac{1}{2\kappa \tau} \int_{]-\kappa \tau,\kappa \tau[} \Bigl(\int_{G_i} \abs{\nabla_b u(e^{tY_i}h)}^Q\dif \nu (h)\Bigr)^\frac{1}{Q} \dif t\dif \tau\\
&=QC' \lambda^\frac{1}{Q} M(I)(0)\;.
\end{split}
\]

Now we prove \eqref{ineqLipschitzApprox}. First one notes that for $g \in G_i$, $t \in \R$,
\begin{equation}
\label{eqYiulambda}
  Y_i u_\lambda (ge^{tY_i})=\tfrac{t}{\sqrt{\lambda^2+t^2}} (u \ast \tfrac{\partial}{\partial t} I_{\sqrt{\lambda^2+t^2}} \eta) (g)\;,
\end{equation}
which can be estimated as above:
\[
  \abs{Y_i u_\lambda (ge^{tY_i})} \le C \frac{t}{(\lambda^2+t^2)^{\frac{1}{2}-\frac{1}{2Q}}} M(I)(0) \le C \lambda^{\frac{1}{Q}-1} M(I)(0)
\]

Now, assume $j \ne i$. Since $G_i$ is normal, $e^{tY_i} e^{sY_j}e^{-tY_i} \in G_i$ for every $s \in \R$, whence $Y_j u_\lambda (ge^{tY_i}e^{sY_j})=(u \ast  I_{\sqrt{\lambda^2+t^2}} \eta )(ge^{tY_i} e^{sY_j}e^{-tY_i})$ and
\[
\begin{split}
  Y_j u_\lambda (ge^{tY_i})
             &=  (\Ad ({e^{tY_i}}) Y_j)(u \ast  I_{\sqrt{\lambda^2+t^2}} \eta )(g)\\
             &=  \bigl(u \ast (\Ad ({e^{tY_i}}) Y_j) I_{\sqrt{\lambda^2+t^2}} \eta \bigr)(g)\\
             &= \bigl(u  \ast (\tfrac{1}{t}\delta_t \Ad ({e^{Y_i}}) Y_j) I_{\sqrt{\lambda^2+t^2}} \eta \bigr)(g)\;.
\end{split}
\]
By Lemma~\ref{lemmaJerison31p}, this can be rewritten as
\begin{equation}
\label{eqYjulambda}
\begin{split}
Y_j u_\lambda(ge^{tY_i})&=\sum_{j=1}^p \sum_{k=1}^{m} u \ast  \frac{t^{j-1}}{(\lambda^2+t^2)^\frac{j}{2}}Y_k^R I_{\sqrt{\lambda^2+t^2}} D^{(k)}_j\eta\;\\
&=-\sum_{j=1}^p \sum_{k=1}^{m}  \frac{t^{j-1}}{(\lambda^2+t^2)^\frac{j}{2}}Y_ku \ast  I_{\sqrt{\lambda^2+t^2}} D^{(k)}_j\eta\;.
\end{split}
\end{equation}
where $\eta^{k}_j=D^{(k)}_j \eta$ is given by Lemma~\ref{lemmaJerison31p}.
Estimating each term as previously, one obtains
\[
  \abs{Y_j u_\lambda (g e^{t Y_i})}  \le C  \sum_{j=1}^p\frac{t^{j-1}}{(\lambda^2+t^2)^{\frac{j}{2}-\frac{1}{2Q}}} M(I)(0) \le C' \lambda^{\frac{1}{Q}-1}M(I)(0) \;.\qedhere
\]
\end{proof}

\section{Proof of the estimate}
\label{sectionMainProof}

Lemma~\ref{lemmaApproximation} brings us in position to prove Theorem~\ref{mainTheorem}:

\begin{proof}[Proof of Theorem~\ref{mainTheorem}]
Decomposing $\varphi$ and $F$ as $\varphi^i=\dualprod{\varphi}{Y_i}$ and $F=\sum_{i=1}^{m}F_i Y_i$, one has
\[
  \int_{G} \dualprod{\varphi}{F}\dif \mu=\sum_{i=1}^{m} \int_{G} \varphi^i F_i\dif \mu\;.
\]
Fixing now $1 \le i \le m$, one has, 
\[
  \int_{G} \varphi^i F_i\dif \mu=\int_{\R} \int_{G_i} F_i(e^{tY_i}h)\varphi^i(e^{tY_i}h)\dif \nu(h)\dif t\;.
\]
Let us estimate the inner integral. For simplicity, first assume that $t =0$. For every $\lambda > 0$, one has
\[
  \int_{G_i} F_i\varphi^i\dif \nu=\int_{G_i} F_i(\varphi^i-\varphi^i_\lambda)\dif \nu+\int_{G_i} F_i\varphi^i_\lambda\dif \nu\;,
\]
where $\varphi^i_\lambda$ is given by Lemma~\ref{lemmaApproximation}.
On the one hand, one has
\begin{equation}
\label{ineqLinftyEstimate}
\begin{split}
 \int_{G_i} F_i\,(\varphi^i-\varphi^i_\lambda)\dif \nu &\le \Norm{F_i}_{\mathrm{L}^1(G_i)}\Norm{\varphi^i-\varphi^i_\lambda}_{\mathrm{L}^\infty(G_i)} \\
  &\le C\lambda^{\frac{1}{Q}} \Norm{F_i}_{\mathrm{L}^1(G_i)} M(I)(0)\;.
\end{split}
\end{equation}
On the other hand,
\[
\begin{split}
 \int_{G_i} F_i\varphi^i_\lambda\dif \nu&=\int_{G_i} \int_{-\infty}^0 \frac{\partial}{\partial s} \bigl[F_i(he^{sY_i} )\varphi^i_\lambda(he^{sY_i})\bigr]\dif s\dif \nu(h) \\
&=\int_{G_i} \int_{-\infty}^0 Y_i \bigl[F_i(he^{sY_i})\varphi^i_\lambda(he^{sY_i})\bigr]\dif s\dif \nu(h) \\
&=\int_{-\infty}^0\int_{G_i}  \bigl[F_i Y_i \varphi^i_\lambda+\varphi^i_\lambda Y_i F_i\bigr](he^{sY_i})\dif \nu(h)\dif s\;. \\
\end{split}
\]
Since $Y_i F_i=-\sum_{j \ne i} Y_j F_j$, this becomes
\[
  \int_{G_i} F_i(h)\varphi^i_\lambda(h)\dif \nu(h)=\int_{-\infty}^0 \int_{G_i}  \bigl[F_i Y_i \varphi^i_\lambda -\sum_{j \ne i} \varphi^i_\lambda Y_i F_i\bigr](he^{sY_i})\dif \nu(h)\dif s\;.
\]
Since $Y_j \in \mathfrak{g}_i$ when $j \ne i$, and since $\nu$ is right-invariant on $h$, integration by parts on $G_i$ yields
\[
 \int_{G_i} F_i(h)\varphi^i_\lambda(h)\dif \nu(h)=\sum_{j=1}^{m} \int_{-\infty}^0 \int_{G_i} \bigl[F_j Y_j \varphi^i_\lambda\Bigr](he^{sY_i})\dif \nu (h)\dif s\;.
\]
We have thus the bound
\begin{equation}
\label{ineqLipchitzEtimate}
\begin{split}
 \Bigl\lvert\int_{G_i} F_i(h)\varphi^i_\lambda(h)\dif \nu(h)\Bigl\lvert &\le \Norm{F}_{\mathrm{L}^1(G)} \Norm{\nabla_b \varphi^i_\lambda}_{\mathrm{L}^\infty(G)} \\
& \le C \lambda^{\frac{1}{Q}-1} \Norm{F}_{\mathrm{L}^1(G)} M(I)(0)\;.
\end{split}
\end{equation}
Choosing now 
\begin{equation}
\label{equationlambda}
  \lambda=\frac{\Norm{F}_{\mathrm{L}^1(G)}}{\Norm{F_i}_{\mathrm{L}^1(G_i)}}\;,
\end{equation}
one obtains by \eqref{ineqLinftyEstimate} and \eqref{ineqLipchitzEtimate}
\[
   \Bigl\lvert\int_{G_i} F_i\varphi^i\dif \nu\Bigr\rvert \le C \Norm{F}_{\mathrm{L}^1(G)}^\frac{1}{Q} \Norm{F_i}_{\mathrm{L}^1(G_i)}^{1-\frac{1}{Q}} M(I)(0)\;.
\]
By translation of this inequality we obtain, for every $t \in \R$,
\[
 \Bigl\lvert\int_{G_i} F_i(e^{tY_i} h)\varphi^i(e^{tY_i} h)\dif \nu(h)\Bigr\rvert \le C \Norm{F}_{\mathrm{L}^1(G)}^\frac{1}{Q} \Norm{F_i}_{\mathrm{L}^1(e^{tY_i} G_i)}^{1-\frac{1}{Q}} M(I)(t)\;.
\]
Integrating this inequality on $\R$, one obtains by H\"older's inequality
\[
\begin{split}
 \Bigl\lvert\int_G F_i \varphi^i \dif \mu\Bigl\lvert&\le \int_{-\infty}^\infty \Bigl\lvert\int_{G_i} F_i(e^{tY_i} h)\varphi^i(e^{tY_i} h)\dif \nu(h)\Bigr\rvert\dif t \\
                          &\le C \Norm{F}_{\mathrm{L}^1(G)}^\frac{1}{Q} \Bigl( \int_{-\infty}^\infty \Norm{F_i}_{\mathrm{L}^1(e^{tY_i} G_i)} \dif t \Bigr)^{1-\frac{1}{Q}}\Bigl(\int_{-\infty}^\infty \bigl[M(I)(t)\bigr]^Q \dif t\Bigr)^\frac{1}{Q} \\
                          &\le C' \Norm{F}_{\mathrm{L}^1(G)} \Norm{\nabla_b \varphi^i}_{\mathrm{L}^Q(G)}\;,
\end{split}
\]
since by the maximal function theorem (see e.g.\ \cite{Stein1993}), there exists $C'' <\infty$ such that 
\[
  \Norm{M(I)}_{\mathrm{L}^Q(\R)} \le C'' \Norm{I}_{\mathrm{L}^Q(\R)}\;.\qedhere
\]
\end{proof}

\section{Elliptic regularity}
\label{sectionRegularity}

Theorem~\ref{theoremRegularity} follows from Theorem~\ref{mainTheorem} and the theory of regularity on homogeneous groups.

\begin{proof}[Proof of Theorem~\ref{theoremRegularity}]
By Theorem~\ref{mainTheorem}, one can write $F_i=\sum_{k=1}^{m} Y_kh_{ki}$, with 
\[
  \Norm{h_{ki}}_{\mathrm{L}^{Q/(Q-1)}(G)} \le C \Norm{F}_{\mathrm{L}^1(G)}\;. 
\]
Therefore, 
\[
  Y_j U_i=Y_j \mathcal{G} \ast \sum_{k=1}^{m} Y_kh_{kj}=\sum_{k=1}^{m} Y_jY_k(\mathcal{G} \ast h_{ij})
\]
where $\mathcal{G}$ is the fundamental solution of $-\Delta_b$. By the analogue of the Calderon--Zygmund inequality for homogeneous groups \cite{Folland1975,RotschildStein1976,FollandStein1974},
\[
\Norm{Y_jY_k(\mathcal{G} \ast h_{ij})}_{\mathrm{L}^{Q/(Q-1)}(G)} \le C\sum_{k=1}^{m}\Norm{h_{kj}}_{\mathrm{L}^{Q/(Q-1)}(G)}
\le C' \Norm{F}_{\mathrm{L}^1(G)}\;.
\]
This concludes the proof.
\end{proof}

\section{Further inequalities}
\label{sectionFurther}

\subsection{$\mathrm{L}^1$--divergence}
Theorem~\ref{mainTheorem} can be extended to the case where the divergence of $F$ is in $\mathrm{L}^1$:

\begin{theorem}
\label{theoremL1Div}
If $\varphi \in C^\infty_c(G,T^*G)$ is a section of the cotangent bundle and the vector field $F \in \mathrm{L}^1(G;T_bG)$ and $\Div_b F=f \in \mathrm{L}^1(G)$ in the weak sense, i.e. 
\[
  \int_G F\psi\dif \nu=-\int_G f\psi\dif \nu
\]
then
\[
   \Bigl\lvert\int_{G} \dualprod{\varphi}{F} \dif \mu\Bigr\rvert \le C (\Norm{F}_{\mathrm{L}^1(G)}\Norm{\nabla_b \varphi}_{\mathrm{L}^Q(G)}+\Norm{\Div_b F}_{\mathrm{L}^1(G)}\Norm{\varphi}_{\mathrm{L}^Q(G)})\;.
\]
\end{theorem}

This version of the inequality is more stable. It can thus be localized by multiplication by cutoff functions. In particular, that under the assumptions of Theorem~\ref{mainTheorem}\;, if $G$ is a multiply connected Lie group, one has the inequality
\[
  \int_{G} \dualprod{\varphi}{F} \le C \Norm{F}_{\mathrm{L}^1(G)}(\Norm{\varphi}_{\mathrm{L}^Q(G)} +\Norm{\nabla_b \varphi}_{\mathrm{L}^Q(G)})\;.
\]

\begin{proof}[Sketch of the proof of Theorem~\ref{theoremL1Div}]
The proof follows the strategy of the proof of Theorem~\ref{mainTheorem} and requires the following refinement in  Lemma~\ref{lemmaApproximation}:
\[
  \Norm{u_\lambda}_{\mathrm{L}^\infty(G)} \le C \lambda^{\frac{1}{Q}-1} M(J)(0)\;,
\]
where
\[
  J(t)=\Bigl(\int_{G_i} \abs{u(e^{tY_i}h)}^Q \dif \nu(h)\Bigr)^\frac{1}{Q}\;.
\]
One obtains in place of \eqref{ineqLipchitzEtimate}
\begin{multline*}
 \Bigl\lvert\int_{G_i} F_i(h)\varphi^i_\lambda(h)\dif \nu(h)\Bigl\lvert  \\
 \le C \lambda^{\frac{1}{Q}-1}\bigl(\Norm{F}_{\mathrm{L}^1(G)} M(I)(0)+ \Norm{\Div_b F}_{\mathrm{L}^1(G)} M(J)(0)\bigr)\;.
\end{multline*}
Choosing again $\lambda$ given by \eqref{equationlambda}, one has
\begin{multline*}
   \Bigl\lvert\int_{G_i} F_i(h)\varphi^i(h)\dif \nu(h)\Bigr\rvert  \\
\le C \Norm{F_i}_{\mathrm{L}^1(G_i)}^{1-\frac{1}{Q}} \Bigl(\Norm{F}_{\mathrm{L}^1(G)}^\frac{1}{Q}  M(I)(0) + \frac{\Norm{\Div_b F}_{\mathrm{L}^1(G)}}{\Norm{F}_{\mathrm{L}^1(G)}^{(Q-1)/Q}} M(J)(0)  \Bigr)\;.
\end{multline*}
One concludes then as in the proof of Theorem~\ref{mainTheorem}.
\end{proof}

\subsection{Fractional spaces}

In Theorem \ref{mainTheorem}, we can also replace  $\Norm{\nabla_b \varphi}_{\mathrm{L}^Q(G)}$ by a fractional Sobolev--Slobodetski\u\i{} norm. In order to define the latter, the group $G$ is endowed by a norm function $\rho \colon G \to \R^+$ such that 
\begin{align*}
  \rho( \delta_t g)&=t \rho(g)\;,\\
  \rho(gh) &\le c (\rho (g)+\rho(h))\;,\\
  \rho(g^{-1}) &\le c \rho(g)\;,
\end{align*}
for some constant $c>0$ (see e.g.\ \cite[Chapter XIII, 5.1.3]{Stein1993}). One can choose for example
\[
 \rho(g)=\inf \{ \lambda > 0 \st g \in B(e,\lambda) \}\;.
\]

\begin{definition}
Let $u \in \mathrm{L}^1_{\mathrm{loc}}(G)$ and $0 < \alpha < 1$. We say that $u \in \mathrm{\dot{S}}^{\alpha,q}(G)$ if
\[
  \Norm{u}_{\mathrm{\dot{S}}^{\alpha,q}}^q=\int_{G} \int_{G} \frac{\abs{u(h)-u(g)}^q}{\rho(g^{-1}h)^{Q+\alpha q}}\dif \mu(g)\dif \mu(h)<+\infty\;.
\]
\end{definition}

The generalization of Theorem \ref{mainTheorem} to fractional spaces is 

\begin{theorem}
\label{theoremFract}
Let $\alpha \in ]0,1[$ and $p\ge 1$ be such that $\alpha q=Q$. 
There exists $C_{\alpha,q}>0$ such that if $\varphi \in C^\infty_c(G,T^*G)$ is a section of the cotangent bundle and the vector field $F \in \mathrm{L}^1(G;T_bG)$ is divergence-free, then
\[
  \Bigl\lvert\int_{G} \dualprod{\varphi}{F}  \dif \mu\Bigr\rvert \le C_{\alpha,q} \Norm{F}_{\mathrm{L}^1(G)}\Norm{\varphi}_{\mathrm{\dot{S}}^{\alpha,q}(G)}\;.
\]
\end{theorem}	

The new ingredient needed to prove Theorem~\ref{theoremFract} is 

\begin{lemma}
\label{lemmaApproximationFracts}
Let $\alpha \in ]0,1[$ and $q \ge 1$. If $\alpha q > Q-1$, there exists $C_{\alpha,q} >0$ such that, for every $u \in C^\infty_c(G)$, $\lambda > 0$, and $1 \le i \le m$, there exists $u_\lambda \in C^\infty(G)$ such that 
\begin{align}
  \label{ineqLinftyApproxFract} \Norm{u-u_\lambda}_{\mathrm{L}^\infty(G_i)} &\le C_{\alpha,q} \lambda^{\alpha-\frac{Q-1}{q}}M(I_{\alpha,q})(0)\;.\\
  \label{ineqLipschitzApproxFract} \Norm{\nabla_b u_\lambda}_{\mathrm{L}^\infty(G)} &\le C_{\alpha,q} \lambda^{\alpha-\frac{Q-1}{q}-1} M(I_{\alpha,q})(0)\;.
\end{align}
where
\[
  I_{\alpha,q}(t)=\Bigl(\int_{G_i}\int_G \frac{\abs{u(e^{tY_i}h)-u(g)}^q}{\rho(g^{-1}h)^{Q+\alpha q}}\dif \mu(g)\dif \nu(h)\Bigr)^\frac{1}{q}\;.
\]
\end{lemma}

\begin{proof}
Define $u_\lambda$ as in Lemma~\ref{lemmaApproximation}.
In order to check \eqref{ineqLinftyApproxFract}, we estimate $u_\lambda(g)-u(g)$ for $g \in G_i$. One has clearly 
\[
  u_\lambda(g)-u(g)=\int_0^{\lambda} \frac{\partial }{\partial \tau}\Bigl[ u \ast I_\tau \eta\Bigr](g) \dif \tau
                =\int_0^{\lambda} \Bigl[u \ast \frac{\partial }{\partial \tau} I_\tau \eta\Bigr] (g)\dif \tau\;.
\]
One writes now
\[
  \frac{\partial }{\partial \tau} I_\tau \eta=\frac{1}{\tau} I_\tau \Tilde{\eta}\;,
\]
where
\[
  \Tilde{\eta}=\frac{\partial }{\partial \tau} I_\tau\eta\Big\vert_{\tau=1}\;.
\]
Note that 
\[
  \int_G \Tilde{\eta}\dif \mu=\frac{d}{d\tau} \int_G I_\tau \eta\dif \mu=\frac{d}{d\tau}1=0\;.
\]

This brings us to
\[
\begin{split}
  u_\lambda(g)-u(g)&= \int_0^{\lambda}\int_{G} u(h)\, \frac{1}{\tau} I_\tau \Tilde{\eta}(h^{-1}g)\dif \mu(h)\dif \tau \\
  &= \int_0^{\lambda}\int_{G}\frac{1}{\mu (B(g,\tau))}\int_{B(g,\tau)} [u(h)-u(k)]\\
&\hspace{10em}\frac{1}{\tau} I_\tau \Tilde{\eta}(h^{-1}g)\dif \mu(k)\dif \mu(h)\dif \tau\;.
\end{split}
\]
Thus, for some $K > 0$,
\[
\begin{split}
  \vert u_\lambda(g)&-u(g)\vert \\
 &\le C \int_0^{\lambda} \frac{1}{\tau^{2Q+1}} \biggl(\int_{B(g,K\tau)}\int_{B(g,\tau)} \abs{u(h)-u(k)}\dif \mu(k)\dif \mu(h)\biggr)\dif \tau \\
&\le C' \int_0^{\lambda} \frac{\tau^{\alpha+\frac{Q}{q}}}{\tau^{2Q+1}} \biggl(\int_{B(g,K\tau)}\int_{B(g,\tau)} \frac{\abs{u(h)-u(k)}}{\rho(k^{-1}h)^{\frac{Q}{q}+\alpha}}\dif \mu(k)\dif \mu(h)\biggr)\dif \tau\;.
\end{split}
\]
Now note that $B(g,K\tau) \cap e^{tY_i}G_i = \emptyset$ 
when 
$\abs{t} 
\ge 
\kappa \tau$, 
for some $\kappa < \infty$. Therefore
\[
\begin{split}
  \abs{u_\lambda(g)-u(g)} 
&\le C \int_0^{\lambda} \frac{\tau^{\alpha+\frac{Q}{q}}}{\tau^{2Q+1}} \int_{]-\kappa \tau,\kappa \tau[} \int_{G_i \cap e^{-tY_i} B(g,K\tau)}\\
&\hspace{8em}\int_{B(g,\tau)} \frac{\abs{u(h)-u(k)}}{\rho(k^{-1}h)^{\frac{Q}{q}+\alpha}}\dif \mu(k) \dif \nu (h) \dif t\dif \tau\;.
\end{split}
\]
Since $\nu(e^{-tY} B(g,K\tau) \cap G_i)\le C s^{Q-1}$, we obtain, by H\"older's inequality,
\[
\begin{split}
  \abs{u_\lambda(g)-u(g)} 
&\le C'\int_0^{\lambda} \frac{\tau^{\alpha+\frac{Q}{q}+(2Q-1)(1-\frac{1}{q})}} {\tau^{2Q}}\frac{1}{2\kappa \tau} \\
&\hspace{2em}\int_{]-\kappa \tau,\kappa \tau[} \Bigl(\int_{G_i} \int_{B(g,\tau)} \frac{\abs{u(h)-u(k)}^q}{\rho(k^{-1}h)^{Q+\alpha q}}\dif \mu(k)\dif \nu (h)\Bigr)^\frac{1}{q} \dif t\dif \tau\\
&=C'' \lambda^{\alpha-\frac{Q-1}{q}} M(I_{\alpha,q})(0)\;.
\end{split}
\]
(The condition $\alpha q > Q-1$ was used to integrate $\tau^{\alpha-\frac{Q-1}{q}-1}$.)
The proof of \eqref{ineqLipschitzApproxFract} is similar.
\end{proof}

The method above also works if we define the Sobolev spaces of fractional order using the Triebel--Lizorkin definition \cite{HanSwayer1994,VagiGatto1999}.

\subsection{Higher order conditions}
In the Euclidean case, estimates similar to Theorem~\ref{mainTheorem} still hold when the condition on the divergence is replaced by a condition on higher-order derivatives \cite{VanSchaftingenHigher}. The same ideas apply to homogeneous groups.

We consider sections given by maps $F : G \to \bigotimes^k T_bG$, where $\bigotimes^k$ is the tensor product. These sections can be identified as differential operators of order $k$, given by
\[
  Fu(g)=\sum_{i_1,\dotsc,i_k \in \{1, \dotsc,m\}} F_{i_1\dotsm i_k} (g)\, (Y_{i_1} \dotsm Y_{i_k} u)(g)\;.
\]
We shall call such sections $k$--order differential operators. 
Now we consider $\bigotimes^k T_b^*G=(\bigotimes^k T_bG)^*$, and $\Sym (\bigotimes^k T_b^*G)$, the vector subspace of $\bigotimes^k T_b^*G$ consisting of tensors which are invariant under the action of the symmetric group $S_k$.

\begin{theorem}
\label{theoremHigher}
Let $k \ge 1$, $F \in \mathrm{L}^1(G;\bigotimes^kT_bG)$ and $\varphi \in C^\infty_c(G,\Sym(\bigotimes^k T_b^*G))$. If for every $\psi \in C^\infty_c(G)$, 
\[
  \int_{G} F \psi \dif \mu =0\;,
\]
then,  
\[
  \Bigl\lvert\int_{G} \dualprod{\varphi}{F}\dif \mu\Bigr\rvert \le C_k \Norm{F}_{\mathrm{L}^1(G)}\Norm{\nabla_b \varphi}_{\mathrm{L}^Q(G)}\;.
\]
\end{theorem}

A restriction appears in the statement of Theorem~\ref{theoremHigher}: for every $g \in G$,  $\varphi(g)$ should be a \emph{symmetric} $k$--linear form. On $\R^n$ this restriction is not really restrictive, since all vector fields commute, so that every $k$--order differential operator is symmetric. This is not any more the case on a noncommutative group, hence the question arises whether the restriction to symmetric $k$--linear forms is essential. In the particular setting of the three-dimensional Heisenberg group this gives:

\begin{openproblem}
Consider the Heisenberg group $\Hei^1$, which is a three-dimensional homogeneous group such that $X=Y_1$, $Y=Y_2$ and $T=[X,Y]$. Assume that $F_i \in \mathrm{L}^1(\Hei^1)$, for $1 \le i \le 4$. If
\[
  T F_1+X^2F_2+Y^2F_3+(XY+YX)F_4=0,
\]
then, by Theorem~\ref{theoremHigher}, $F_i \in \mathrm{\dot{S}}^{-1,4/3}(\Hei^1)$, for $i=2,3,4$. Does one also have $F_1 \in \mathrm{\dot{S}}^{-1,4/3}(\Hei^1)$?
\end{openproblem}

The next Lemma is the essential step in the proof of Theorem~\ref{theoremHigher}.

\begin{lemma}
\label{lemmaHigher}
If $k \ge 1$, $(F_{i_1 \cdots i_k})_{1 \le i_l \le m} \in \mathrm{L}^1(G)$ and
\begin{equation}
\label{condHigherOrder}
 \sum_{i_1,\dotsc,i_k \in \{1, \dotsc,m\}}Y_{i_1} \dots Y_{i_k} F_{i_1 \dots i_k}=0\;,
\end{equation}
then for every $u \in C^\infty_c(G)$,
\[
  \Bigl\lvert\int_{G} F_{1 \dotsm 1}u  \dif \mu\Bigr\rvert \le C_k\Norm{F}_{\mathrm{L}^1} \Norm{\nabla_b u}\;.
\]
\end{lemma}

\begin{proof}[Proof of Theorem~\ref{theoremHigher}]
Since $k$--linear symmetric forms $\omega$ of the form 
\[
  \omega(X_1,\dotsc,X_k)=X^*(X_1)\dotsm X^*(X_k),
\]
for $X_i \in \mathfrak{g}$ and $X^* \in \mathfrak{g}^*$ generate the finite-dimensional space $\Sym(\bigotimes^k \mathfrak{g}^*)$, it is sufficient to prove a similar estimate for every $\varphi(g;X_1,\dotsc,X_k)=u(g)X^*(X_1)\dotsm X^*(X_k)$. Without loss of generality, we can assume that the kernel of $X^*$ is spanned by $Y_2,\dotsc,Y_m$, and that $X^*(Y_1)=1$. Writing $F$ as
\[
  F=\sum_{i_1,\dotsc,i_k \in \{1, \dotsc,m\}}F_{i_1 \dotsm i_k}Y_{i_1} \dotsm Y_{i_k}\;,
\]
with $F_{i_1 \dotsm i_k}\in \mathrm{L}^1(G)$, one obtains
\[
  \int_{G} \dualprod{\varphi}{F}\dif\mu=\int_{G} uF_{1 \dotsm 1}\dif\mu\;.
\]
The functions $F_{i_1\dotsm i_k}$ satisfy the assumptions of Lemma~\ref{lemmaHigher}, which yields the conclusion.
\end{proof}

We now have to prove Lemma~\ref{lemmaHigher}. The main ingredient is an improvement of Lemma~\ref{lemmaApproximation} in which the decay of higher-order derivatives of $u_\lambda$ is controlled.

\begin{lemma}
\label{lemmaApproximationHigher}
There exists $C >0$ such that, for every $u \in C^\infty_c(G)$, $\lambda > 0$, and $1 \le i \le m$, there exists $u_\lambda \in C^\infty(G)$ such that for every $t \in \R$ 
\begin{align}
  \label{ineqLinftyApproxHigher} \Norm{u-u_\lambda}_{\mathrm{L}^\infty(G_i)} &\le C \lambda^{\frac{1}{Q}}M(I)(0)\;,\\
  \label{ineqHigherApprox} \Norm{\nabla_b^k u_\lambda}_{\mathrm{L}^\infty(G_i e^{tY_i})} &\le \frac{C_k}{ \bigl(\sqrt{\lambda^2+t^2}\bigr)^{k-\frac{1}{Q}}} M(I)(0)\;.
\end{align}
\end{lemma}

\begin{proof}[Proof of Lemma~\ref{lemmaApproximationHigher}]
Define $u_\lambda$ as in the proof of Lemma~\ref{lemmaApproximation}. One still has  \eqref{ineqLinftyApproxHigher}.

Now let us prove \eqref{ineqHigherApprox}. Let ${i_1},\dotsc, {i_k}\in \{1,\dotsc,m\}$. One has
\[
  Y_{i_1} \dotsm Y_{i_k} u_\lambda(ge^{tY_i})
  = (u \ast \eta^t_{i_1 \dotsm i_k})(g)\;,
\]
where $\eta^t_{i_1 \dotsm i_l}$ is defined recursively by $\eta^t=I_{\sqrt{\lambda^2+t^2}} \eta$ and
\[
  \eta^t_{i_1 \dotsm i_{l+1}}
=\begin{cases}
   \frac{\partial}{\partial t} \eta^t_{i_2 \dotsm i_{l+1}} & \text{if $i_{1}=i$,}\\
   [\Ad ({e^{tY_i}})Y_{i_1}] \eta^t_{i_2\cdots i_{l+1}} & \text{if $i_{1}\ne i$.}
 \end{cases}
\]
We now claim that for every $l \ge 0$ and $i_1,\dotsc, i_l \in \{1, \dotsc,m\}$, there exists $q \ge 1$, $\eta_r^{(j)} \in C^\infty_c(G)$ and $\theta_r \in C^\infty(\R^+)$, with $1 \le r \le q$ and $1 \le j \le m$ such that 
\begin{equation}
\label{etaRepresentation}
   \eta^t_{i_1 \dotsm i_l}=\sum_{\substack{0 \le r \le q\\ 1 \le j \le m}} \theta_r(t) I_{\sqrt{\lambda^2+t^2}} Y_j^R\eta^{(j)}_{p}\;,
\end{equation}
where 
\[
  \theta_r^{(k)} (t)\le \frac{C_{i_1 \dotsm i_l,r,k}}{(\lambda^2+t^2)^\frac{r+k}{2}}\;.
\]
Note that the constants $C_{i_1 \dotsm i_l,r,k}$ are independent of $t$ and $\lambda$.

Indeed, for $l=1$, \eqref{etaRepresentation} follows respectively from \eqref{eqYiulambda} together with Lemma~\ref{lemmaJerison31p}, and from \eqref{eqYjulambda}.
Assume now that \eqref{etaRepresentation} holds for $l \ge 1$. One has in particular
\begin{equation*}
   \eta^t_{i_2 \dotsm i_{l+1}}=\sum_{\substack{0 \le r \le q\\ 1 \le j \le m}} \theta_r I_{\sqrt{\lambda^2+t^2}} Y_j^R\eta^{(j)}_{p}\;.
\end{equation*}
If $i_1=i$, one has, by Lemma~\ref{lemmaJerison31p},
\[
\begin{split}
 \eta^t_{i_1i_2 \dotsm i_{l+1}}&=\frac{\partial}{\partial t} \sum_{\substack{0 \le r \le q\\ 1 \le j \le m}} \theta_r(t) I_{\sqrt{\lambda^2+t^2}} Y_j^R\eta^{(j)}_r \\
&=\sum_{\substack{0 \le r \le q\\ 1 \le j \le m}} \theta_r'(t) I_{\sqrt{\lambda^2+t^2}} Y_j^R\eta^{(j)}_r+\sum_{\substack{0 \le p \le q\\ 1 \le j \le m}} \theta_r(t) \frac{\partial }{\partial t} I_{\sqrt{\lambda^2+t^2}} Y_j^R\eta^{(j)}_r\\
&=\sum_{\substack{0 \le r \le q\\ 1 \le j \le m}} \theta_r'(t) I_{\sqrt{\lambda^2+t^2}} Y_j^R\eta^{(j)}_r+\sum_{\substack{0 \le r \le q\\  1 \le k \le m}} \theta_r(t) \frac{t}{\lambda^2+t^2} I_{\sqrt{\lambda^2+t^2}} Y_k^R \tilde{\eta}_r^{(k)}\;,
\end{split}
\]
where
\[
  \tilde{\eta}_r^{(k)}=D^{(k)}\sum_{1 \le j \le m} Y_j^R\eta^{(j)}_r\;,
\]
from which \eqref{etaRepresentation} follows.
If $i_1 \ne i$, by Lemma~\ref{lemmaJerison31p} again
\[
\begin{split}
 \eta^t_{i_1i_2 \cdots i_{l+1}}&=\tfrac{1}{t}[\delta_t \Ad ({e^{Y_i}}) Y_{i_1}] \sum_{\substack{0 \le r \le q\\ 1 \le j \le m}} \theta_r(t) I_{\sqrt{\lambda^2+t^2}} Y_j^R\eta^{(j)}_r \\
&=\sum_{\substack{0 \le r \le q\\ 1 \le j \le m}} \theta_r(t) I_{\sqrt{\lambda^2+t^2}} Y_j^R\Tilde{\eta}^{(j)}_r\;,
\end{split}
\]
where
\[
 \Tilde{\eta}^{(j)}_r=\sum_{\substack {1 \le k \le m\\ 1 \le s \le p}} \frac{t^{s-1}}{(\lambda^2+t^2)^\frac{s}{2}} D^{(j)}_{i_1}  Y^R_k \eta_r^k\;.
\]
Thus \eqref{etaRepresentation} is established, and brings us in position to conclude as in the proof of Lemma~\ref{lemmaApproximation} that 
\[
  \abs{X_{i_1} \dotsm X_{i_j} u_\lambda(ge^{tY_i})}
  = \abs{(u \ast \eta^t_{i_1 \dotsm i_k})(g)}
  \le C(\lambda^2+t^2)^{\frac{1}{2Q}-\frac{k}{2}} M(I)(0)\;.\qedhere
\]
\end{proof}

We end this section by the proof of Lemma~\ref{lemmaHigher}.

\begin{proof}[Proof of Lemma~\ref{lemmaHigher}]
As in the proof of Theorem~\ref{mainTheorem}, we need to estimate
\[
  \int_{G_1} F_{1 \dotsm 1}u_\lambda\dif \nu
\]
were $u_\lambda$ is now given by Lemma~\ref{lemmaApproximationHigher} instead of Lemma~\ref{lemmaApproximation}.
One has 
\[
\begin{split}
  \int_{G_1} F_{1 \dotsm 1}u_\lambda\dif \nu=&\int_{G_1}\int_{-\infty}^0 F_{1 \dotsm 1}(he^{sY_1}) \frac{\partial^k}{\partial s^{k}}\Bigl[\frac{s^{k-1}}{(k-1)!} u_\lambda(he^{sY_1})\Bigr]\dif s\dif \nu(h) \\
  &+(-1)^k\int_{G_1} \int_{-\infty}^0 \frac{s^k}{k!} u_\lambda(he^{sY_1})\frac{\partial^k}{\partial s^k}F_{1 \dotsm 1}(he^{sY_1})\dif s\dif \nu(h)\;.
\end{split}
\]
The first term gives
\begin{multline*}
\int_{G_1}\int_{-\infty}^0 F_{1 \dotsm 1}(he^{sY_1}) \frac{\partial^k}{\partial s^{k}}\Bigl[\frac{s^{k-1}}{(k-1)!} u_\lambda(he^{sY_1})\Bigr]\dif s\dif \nu(h)\\
=\sum_{l=0}^k \binom{k}{l}\int_{G_1}\int_{-\infty}^0 F_{1 \dotsm 1}(he^{sY_1}) \frac{s^{l-1}}{(l-1)!} Y_1^lu_\lambda(he^{sY_1})\dif s\dif \nu(h)\;.
\end{multline*}
By Lemma~\ref{lemmaApproximationHigher}, one has
\begin{multline*}
 \int_{G_1}\int_{-\infty}^0 F_{1 \dotsm 1}(he^{sY_1}) \frac{s^{l-1}}{(l-1)!} Y_1^lu_\lambda(he^{sY_1})\dif s\dif \nu(h) \\
 \le C \int_{-\infty}^0 \Norm{F_{1\dots 1}}_{\mathrm{L}^1(e^{sY_1} G_1)} \frac{s^{l-1}}{(\lambda^2+s^2)^{\frac{l}{2}-\frac{1}{2Q}}}M(I)(0)\dif s \\
 \le C' \lambda^{\frac{1}{Q}-1}\Norm{F_{1 \dotsm 1}}_{\mathrm{L}^1(G)} M(I)(0)\;.
\end{multline*}
For the other term, by the assumption \eqref{condHigherOrder}, one has
\begin{multline*}
 \int_{G_1} \int_{-\infty}^0 \frac{s^{k-1}}{(k-1)!} u_\lambda(he^{sY_1})\frac{\partial^k}{\partial s^k}F_{1 \dotsm 1}(he^{sY_1})\dif s\dif \nu(h) \\
=-\sum_{(i_1,\cdots,i_k)\ne (1, \dots,1)}  \int_{G_1} \int_{-\infty}^0\frac{s^{k-1}}{(k-1)!}u_\lambda(he^{sY_1})\\
Y_{i_1}\dots Y_{i_k} F_{i_1 \dots i_k}(he^{sY_1})\dif \nu(h)\dif s\;.
\end{multline*}
One has then 
\begin{multline*}
 \int_{G_1} \int_{-\infty}^0\frac{s^{k-1}}{(k-1)!}u_\lambda(he^{sY_1})Y_{i_1}\dots Y_{i_k} F_{i_1 \dots i_k}(he^{sY_1})\dif \nu(h)\dif s\\
 =(-1)^k \int_{G_1} \int_{-\infty}^0F_{i_1 \dots i_k}(he^{sY_1}) \Hat{Y}_{i_k}\dots \Hat{Y}_{i_1}\Bigl[\frac{s^{k-1}}{(k-1)!}u_\lambda\Bigr](he^{sY_1}) \dif \nu(h)\dif s\;,
\end{multline*}
where $\Hat{Y}_j=\frac{\partial}{\partial s}+Y_1$ if $j=1$ and $\Hat{Y}_j=Y_j$ otherwise. One obtains then as previously
\begin{multline*}
\Bigl\lvert\int_{G_1} \int_{-\infty}^0F_{i_1 \dots i_k}(he^{sY_1}) \Hat{Y}_{i_k}\dots \Hat{Y}_{i_1}\Bigl[\frac{s^{k-1}}{(k-1)!}u_\lambda\Bigr](he^{sY_1}) \dif \nu(h)\dif s \Bigr\rvert  \\
\le C \Norm{F}_{\mathrm{L}^1(G)}M(I)(0) \lambda^{\frac{1}{Q}-1}\;.
\end{multline*}
The proof ends as the proof of Theorem~\ref{mainTheorem}.
\end{proof}

\providecommand{\bysame}{\leavevmode\hbox to3em{\hrulefill}\thinspace}

\end{document}